\documentclass[12pt,a4paper]{article}

\usepackage{a4wide}
\usepackage{diagbox}
\usepackage{mathtools}
\usepackage{latexsym}
\usepackage{epsfig}
\usepackage{amsmath,amsthm,amssymb,enumerate}

\usepackage{hyperref}

\usepackage{color}

\def\nn{\nonumber}
\def\a{\alpha}   \def\D{\Delta}
\def\e{\varepsilon}    
  \def\k{\kappa}
 \def\th{\theta}    
 \def\m{\mu} \def\n{\nu} 
   
 \def\om{\omega}

\def\cQ{{\cal Q}}

\def\A{\forall}
\newtheorem{theorem}{Theorem}
\newtheorem{lemma}[theorem]{Lemma}

\newtheorem{claim}{Claim}

\newtheorem{conjecture}[theorem]{Conjecture}

\newcommand{\rdup}[1]{{\left\lceil #1\right\rceil }}
\newcommand{\rdown}[1]{{\left\lfloor #1\right \rfloor}}

\newcommand{\brac}[1]{\left(#1\right)}

\newcommand{\bfrac}[2]{\left(\frac{#1}{#2}\right)}

\newcommand{\rai}{\rightarrow \infty}
\newcommand{\ra}{\rightarrow}

\allowdisplaybreaks
\newcommand{\set}[1]{\left\{#1\right\}}
\def\sm{\setminus}

\def\E{\mathbb{E}}

\def\Pr{\mathbb{P}}
\def\w.h.p.{{\bf w.h.p.}}

\newcommand{\ignore}[1]{}

\def\cL{{\mathcal L}}

\def\cQ{{\mathcal Q}}

\newcommand{\beq}[2]{\begin{align}\label{#1}#2\end{align}}

\def\nn{\nonumber}

\def\Dr{\D_{r-1}}

\usepackage{tikz}
\usetikzlibrary{decorations.pathmorphing}
\usetikzlibrary{positioning}
\usetikzlibrary{arrows,automata}
\usetikzlibrary{shapes.misc}
\usetikzlibrary{backgrounds}
\usetikzlibrary{arrows,shapes}

\def\N{\mathbb{N}}

\usepackage[utf8]{inputenc}
\usepackage{graphicx}
\usepackage[short,nocomma]{optidef}
\usepackage{todonotes}
\usepackage{amsthm}
\usepackage{amssymb}
\usepackage{caption}
\usepackage{subcaption}
\usepackage{hyperref}
\usepackage{thm-restate}
\usepackage{cleveref}

\theoremstyle{plain}

\allowdisplaybreaks

\begin{document}
	
\author{
Alan Frieze\thanks{Department of Mathematical Sciences, Carnegie Mellon University, Pittsburgh PA, USA. Research supported in part by NSF grant DMS1952285. \protect\href{mailto:frieze@cmu.edu}{\protect\nolinkurl{frieze@cmu.edu}}}
\and
Ross J. Kang\thanks{Korteweg--de Vries Institute for Mathematics, University of Amsterdam, the Netherlands. Partially supported by the Dutch Research Council (NWO) grant OCENW.M20.009 and the Gravitation Programme NETWORKS (024.002.003) of the Dutch Ministry of Education, Culture and Science (OCW).  \protect\href{mailto:r.kang@uva.nl}{\protect\nolinkurl{r.kang@uva.nl}}}
\and
Aditya Raut\thanks{Department of Mathematical Sciences, Carnegie Mellon University, Pittsburgh PA, USA. \protect\href{mailto:araut@andrew.cmu.edu}{\protect\nolinkurl{araut@andrew.cmu.edu}}}
\and
Michelle Sweering\thanks{CWI, Amsterdam, the Netherlands. Supported by the Gravitation Programme NETWORKS (024.002.003) of NWO.  \protect\href{mailto:michelle.sweering@cwi.nl}{\protect\nolinkurl{michelle.sweering@cwi.nl}}}
\and
Hilde Verbeek\thanks{CWI, Amsterdam, the Netherlands. Supported by the Constance van Eeden PhD
Fellowship. \protect\href{mailto:hilde.verbeek@cwi.nl}{\protect\nolinkurl{hilde.verbeek@cwi.nl}}}
}

\date{}
\title{Coloring powers of random graphs}
\maketitle

\begin{abstract}
Given a graph $G$ and an integer $r\ge 1$, the $r$th power $G^r$ of $G$ is the graph obtained from $G$ by adding edges for all pairs of distinct vertices at distance at most $r$ from each other.
We focus on two basic structural properties of the $r$th power of the binomial random graph $G_{n,p}$, namely, the maximum degree $\Delta(G_{n,p}^r)$ and the chromatic number $\chi(G_{n,p}^r)$, and give with high probability (w.h.p.) bounds.

In the sparse case that $p=d/n$ for some fixed constant $d>0$, we prove the following.
We prove that w.h.p.~$\Delta(G_{n,p}^r) \sim \frac{\log n}{\log_{(r+1)}n}$ (where $\log_{(1)}n=\log n$ and $\log_{(r+1)}n=\log\log_{(r)}n$) and that w.h.p.~$\Delta(G_{n,p}^{\rdown{r/2}})+1 \le \chi(G_{n,p}^r) \le \Delta(G_{n,p}^{r-1})+1$. For $r=2$, we show the upper bound holds with equality.

For denser cases,
for $d$ satisfying $d=\omega(\log n)$ and $d\le n^{1/r-\Omega(1)}$ as $n\to\infty$, we have $\chi(G_{n,p}^r) = \Theta(d^r/\log d)$ w.h.p.
\end{abstract} 

\section{Introduction}

The $r$th power $G^r$ of a graph $G$ is obtained from $G$ by adding edges between all pairs of distinct vertices that are connected by a path of $r$ or fewer edges in $G$. Powers of graphs arise naturally in various contexts, such as, for instance, number theory~\cite{Pok14}, ad hoc communication networks~\cite{AgHa03}, and distributed computing~\cite{Lin92}. 
Note that $G$ has diameter at most $r$ if and only if $G^r$ is a complete graph. We consider the $r$th power $G_{n,p}^r$ of the binomial random graph $G_{n,p}$, which is the probability space of graphs on vertex set $[n]=\{1,\dots,n\}$ in which each of the $\binom{n}{2}$ pairs of vertices is included independently at random with probability $p$ as an edge.
In particular, we focus on two basic properties of $G_{n,p}^r$, namely, its maximum degree $\Delta(G_{n,p}^r)$ and its chromatic number $\chi(G_{n,p}^r)$.

(Recall that the maximum degree $\Delta(G)$ of some graph $G$ is the maximum number of neighbours a vertex in $G$ has, while the chromatic number $\chi(G)$ is the least number of colors needed in a coloring of all vertices so that no two endpoints of a vertex have the same color. We have that $\chi(G) \le \Delta(G)+1$ always.)

We seek good bounds on these parameters that hold with probability tending to $1$ as $n\to\infty$, that is, {\em with high probability (w.h.p.)}.
Let $\log_{(k)} n$ indicate the repeated application of the log-function $k$ times.  So, for example, $\log_{(3)} n = \log\log\log n$. 

In the case $r=1$, we understand quite a lot about these parameters. For example, since the degree of any vertex is a binomial random variable, it is straightforward to show that, if $p=d/n$ for some fixed constant $d>0$, then the maximum degree $\Delta(G_{n,p})$ of $G_{n,p}$ satisfies $\Delta(G_{n,p}) \sim \frac{\log n}{\log_{(2)} n}$ w.h.p.~(see e.g.~\cite[Thm.~3.4]{FrKa16} and~\cite{Bol80}).
Because it is a much more `global' parameter, the chromatic number $\chi(G_{n,p})$ of $G_{n,p}$ has been more difficult to pin down. Indeed, in the sparse regime, despite remarkable advances, see e.g.~\cite{AcNa05,CoVi16,ACG22}, it remains an open problem of Erd\H{o}s and R\'enyi~\cite{ErRe60} to determine the sharp threshold for $k$-colorability, if it exists. Nevertheless, due to celebrated work~\cite{Mat87,Bol88,MaKu90,Luc91} (see~\cite{KaMc15} for further background), we know for instance that w.h.p.
\[
\chi(G_{n,p}) \sim \frac{n}{\alpha(G_{n,p})}\sim \frac{n\log(1/(1-p))}{\log n},
\]
provided $np = \omega(1)$ (where $\alpha(G)$ denotes the independence number of $G$, the size of a largest independent set in $G$).

For $r>1$, it is intuitive that taking into consideration paths of length at most $r$ could introduce additional `non-local' complications. 
We note that the problem of determining the minimum $r$ for which $G_{n,p}^r$ is guaranteed w.h.p.~to be a clique is finely understood, see e.g.~\cite{Bol81,KlLa81,RiWo10,ABG12}.
There has also been much attention for coloring powers of not-necessarily-random graphs, see e.g.~\cite{AlMo02,AgHa03,HHMR08+,AEH13,KaPi16,KaPi18,KaLo19}.
It might be surprising how comparatively little study of $\Delta(G_{n,p}^r)$ and $\chi(G_{n,p}^r)$ for $r>1$ there is in the literature~\cite{AtFr04,GLMP23,FrRa23+}.

For the maximum degree, Garapaty, Lokshtanov, Maji and Pothen~\cite{GLMP23} proved that if $p=d/n$ where $d>0$ is constant, then w.h.p.~the maximum degree of $G_{n,p}^r$ satisfies $\D(G_{n,p}^r)=\Theta_r\left(\frac{\log n}{\log_{(r+1)} n}\right)$. They estimated the implicit multiplicative factor in the range $[0.05\cdot 2^{-r}, 6]$. We strengthen this in the following theorem.

\begin{theorem}\label{th1}
Let $p=d/n$, where $d>0$ is a constant and let $r\geq 1$ be a fixed positive integer. Then, w.h.p.~$\D(G_{n,p}^r) \sim\frac{\log n}{\log_{(r+1)} n}$.
\end{theorem}
\noindent
We prove Theorem~\ref{th1} in Section~\ref{Delta}.

For the chromatic number, we show in the sparse regime, with $p=d/n$ for some fixed constant $d$, that $\chi(G_{n,p}^2)$ is essentially determined by the clique in $G_{n,p}^2$ induced by the vertex of maximum degree in $G_{n,p}$ and all its direct neighbours. 
\begin{theorem}\label{th2}
Let $p=d/n$, where $d>0$ is a constant. Then, w.h.p.~$\chi(G_{n,p}^2) =\D(G_{n,p})+1$.
\end{theorem}
\noindent
We prove this in Section~\ref{chi2}. In Section~\ref{chir}, we partially generalize the result to $r\geq 2$.
\begin{theorem}\label{th3}
Let $p=d/n$, where $d>0$ is a constant and let $r\geq 2$ be a fixed positive integer. Then, w.h.p.~$\D(G_{n,p}^{\rdown{r/2}})\leq \chi(G_{n,p}^r)\leq \D(G_{n,p}^{r-1})+1$.
\end{theorem}
\noindent
Within the context of Theorem~\ref{th3}, it was earlier shown by Garapaty, Lokshtanov, Maji and Pothen~\cite{GLMP23} that $\chi(G_{n,p}^2) = \Theta(\Delta(G_{n,p}))$ w.h.p. Frieze and Raut~\cite{FrRa23+} later proved that the {\em list chromatic number} of $G_{n,p}^2$ satisfies $\chi_\ell(G_{n,p}^2)\sim \D(G_{n,p})$ w.h.p.

We will also show at the other end for a reasonably large range of more moderate densities --- namely, those $d=np$ satisfying both $d =\omega(\log n)$ and $d = n^{1/r-\Omega(1)}$ --- that  $\chi(G_{n,p}^r)$  is essentially determined by the independence number $\alpha(G_{n,p}^r)$ of $G_{n,p}^r$.  (Recall that the independence number $\alpha(G)$ of a graph $G$ is the maximum cardinality of an independent set of $G$, a vertex subset spanning no edges.) In this density regime, Alon and Mohar~\cite{AlMo02} proved that $\alpha(G_{n,p}^r)=O((n \log d)/d^r)$ w.h.p. In Section~\ref{sec:maindense} we will prove the following.
\begin{theorem}\label{th4}
Let $\e$ be an arbitrarily small positive constant and let $r\geq 1$ be a fixed positive integer. Let $d=d(n)$ satisfy  $d=\omega\log n$ where $\omega\to\infty$ and $d\leq n^{1/r-\e}$ as $n\to\infty$, and set $p=p(n)=d/n$. Then $\chi(G_{n,p}^r) = \Theta(d^r/\log d)$ w.h.p.
\end{theorem}
\noindent
One can compare this to the general bound $\chi(G^r) \le \Delta(G)^r+1$ for any $G$, keeping in mind that $\Delta(G_{n,p}) \sim d$ w.h.p.~in the range of edge densities prescribed in \Cref{th4}.  
If $d$ satisfies $d = \omega((n \log n)^{1/r})$, then w.h.p.~$G_{n,p}^r$ is a clique~\cite{KlLa81}, in which case $\chi(G_{n,p}^r)=n$, and so the upper bound in the range for $p$ above is nearly best possible.
Alon and Mohar~\cite{AlMo02} obtained similar estimates for $\chi(G^r)$ when $G$ has sufficiently large girth; see also~\cite{KaPi16,KaPi18,KaLo19}.
\section{Proof of Theorem~\ref{th1}}\label{Delta}
Given $v\in[n]=V(G), G=G_{n,p}$, let $N_t(v)$ denote the set of vertices of $G$ at distance $t$ from $v$. Now it is well-known that $\D(G_{n,p})=o(\log n)$ w.h.p., see for example~\cite{FrKa16}. As a consequence, $|N_t(v)|=o(\log^tn)$ for $t\leq r$. Let $\ell_0=1$ and let $0\leq \ell_t\leq \log^tn$ for $t=1,2,\ldots,r$. Then the probability that $|N_t(v)|=\ell_t,t=1,2,\ldots,r$, satisfies
\begin{align*}
&\Pr(d_i(v)=\ell_i \ \A i =1,2,...,r ) \\
&= \prod_{i=1}^r\binom{n-\ell_0-\ell_1-\cdots-\ell_{i-1}}{\ell_i}(1-(1-p)^{\ell_{i-1}})^{\ell_i} (1-p)^{\ell_{i-1}(n-\ell_0-\ell_1-\cdots-\ell_{r-1})} \\
&=\prod_{i=1}^r\frac{n^{\ell_i}}{\ell_i!}(\ell_{i-1}p)^{\ell_i}e^{-np\ell_{i-1}}\brac{1+O\bfrac{\log^rn}{n}}\\
& = \frac{d^{\ell_1+...+\ell_r} e^{-d(\ell_0+\ell_1+...+\ell_{r-1})}}{\ell_1!\ell_2!...\ell_r!} \cdot \ell_1^{\ell_2}\cdot\ell_2^{\ell_3}\cdots\ell_{r-1}^{\ell_r} \cdot \brac{1+O\bfrac{\log^rn}{n}}.
\end{align*}
So, using the notation $a\sim b$ to mean that $a=(1+o(1))b$ as $n\to\infty$,
\begin{align}
\Pr\brac{\sum_{i=1}^r d_i(v) = D}& = \sum_{\substack{\ell_1,...,\ell_r \\ \ell_1+...+\ell_r = D}} \frac{d^{D} e^{-D(1+D-\ell_r)}}{\ell_1!\ell_2!...\ell_r!} \cdot \ell_1^{\ell_2}\cdot\ell_2^{\ell_3}\cdots\ell_{r-1}^{\ell_r} \cdot \brac{1+O\bfrac{\log^rn}{n}}\nonumber\\
& = \brac{1+O\bfrac{\log^rn}{n}}\sum_{\substack{\ell_1,...,\ell_r \\ \ell_1+...+\ell_r = D}} u_{\ell_1,...,\ell_r},\label{uhu}
\end{align}
where 
\beq{ul}{
u_{\ell_1,...,\ell_r} = \frac{d^De^{-D(1+D-\ell_r)}}{\ell_1!\ell_2!...\ell_r!}\cdot \ell_1^{\ell_2}\cdots \ell_{r-1}^{\ell_r}.
}
Using Stirling's approximation $\log(m!) = m\log m - m + \frac12\log m+O(1)$, we have, where $L=\ell_1\ell_2\cdots\ell_r$,
\beq{bound}{
\log u_{\ell_1,...,\ell_r} = D\log d - D(D-\ell_r) - \sum\limits_{i=1}^r \ell_i \log\frac{\ell_i}{\ell_{i-1}} -\tfrac12\log L+O(r).
}
The following lemma bounds the sum in~\eqref{bound}. 	
\begin{lemma}\label{lem2}
	For $\ell_0=1$ and $\ell_1,...,\ell_r \in \N$ such that $\sum\limits_{i=1}^r \ell_i = D$, we have $\min \sum\limits_{i=1}^r \ell_i \log \frac{\ell_i}{\ell_{i-1}} \geq D\log_{(r)}D - O(D)$, for sufficiently large $D$.
\end{lemma}
\begin{proof}
	We proceed by induction on $r$. For $r=1$, the result holds since we have $\ell_1=D$, implying that $\ell_1\log\frac{\ell_1}{\ell_0} = D\log D$. Assume that the result holds for $r-1$.  

{\bf Case 1: $(D-\ell_r)\log_{(r-1)}(D-\ell_r) \geq D\log_{(r)} d$:}\\
Because $\sum\limits_{i=1}^{r-1} \ell_i = D-\ell_r$, from the induction hypothesis we have $\sum\limits_{i=1}^{r-1} \ell_i\log\frac{\ell_i}{\ell_{i-1}} \geq (D-\ell_r)\log_{(r-1)}(D-\ell_r)$. So this case is done.

{\bf Case 2: $(D-\ell_r)\log_{(r-1)}(D-\ell_r)< D\log_{(r)} D$:}\\
For $D$ sufficiently large, we have $\frac{D}{10}\log_{(r-1)} \frac{D}{10} > D\log_{(r)} D$. Hence $\ell_r \geq \frac{9D}{10}$, implying $\ell_{r-1} \leq \frac{D}{10}$ and $\frac{\ell_r}{\ell_{r-1}} \geq 9$.

We use Lagrange multipliers. But first we deal with the constraints $\ell_i\geq 0$ for $i=1,2,\ldots,r$. If $\ell_i=0$ and $\ell_{i+1}\neq 0$ then $\sum\limits_{i=1}^r \ell_i \log \frac{\ell_i}{\ell_{i-1}}=\infty$. If $\ell_i=\ell_{i+1}=\cdots=\ell_r=0$ then $\sum\limits_{i=1}^r \ell_i \log \frac{\ell_i}{\ell_{i-1}}=\sum\limits_{i=1}^{i-1} \ell_i \log \frac{\ell_i}{\ell_{i-1}}$ and the result follows by induction.  So in effect there is one constraint: $\sum\limits_{i=1}^r \ell_i = D$.

Define $\cL(\ell_1,...,\ell_r,\lambda) = \sum\limits_{i=1}^r \ell_i \log \frac{\ell_i}{\ell_{i-1}} + \lambda \left(\sum\limits_{i=1}^r \ell_i - D\right)$. By the Lagrange multiplier theorem, the minima must satisfy $\frac{\partial\cL}{\partial \ell_i} = 0$ for all $i$. Notice that $\frac{\partial \cL}{\partial \ell_i} = 1 + \log\frac{\ell_{i}}{\ell_{i-1}} - \frac{\ell_{i+1}}{\ell_i} + \lambda$ for $1 \leq i < r$, and $\frac{\partial\cL}{\partial\ell_r} = 1+\log \frac{\ell_r}{\ell_{r-1}} +\lambda$. Let $p_i := \frac{\ell_i}{\ell_{i-1}}$. From $\frac{\partial\cL}{\partial \ell_r}=0$, we have $1+\log p_r + \lambda = 0$ which implies that $p_r = e^{-(\lambda+1)}$. For $1\leq i<r$, from $\frac{\partial\cL}{\partial\ell_i}=0$ we have $1+\log p_i -p_{i+1}+\lambda = 0$ which  implies that $p_i = e^{p_{i+1}-(1+\lambda)} = p_r\cdot e^{p_{i+1}}\geq p_rp_{i+1}$. Thus, we can iteratively obtain the exact expressions for $p_1,p_2,...,p_{r-1}$ in terms of $p_r$. Now recall that $p_r\geq 9$ from induction hypothesis, hence $p_i \geq 9$ for all $i$. 
	
Since $\ell_0=1$, $\ell_1=p_1$. Now $\ell_i = p_i\cdot p_{i-1}\cdots p_1$, for all $i=1,2,...,r$ and then $\frac{\ell_r}{\ell_i} = p_rp_{r-1}...p_{i+1} \geq 9^{r-i}$. Now $\sum\limits_{i=1}^r \ell_i = D$ and so clearly $D > \ell_r$. Moreover, 
\[
D = \ell_r \left(\sum\limits_{i=1}^r \frac{\ell_i}{\ell_r}\right) \leq \ell_r \left(\sum\limits_{i=1}^r \frac{1}{9^{r-i}}\right) < \ell_r\left(\sum\limits_{j=0}^\infty \frac{1}{9^j}\right) = \frac{9}{8}\cdot \ell_r. 
\]
We can thus assume that $D = c_1 \cdot \ell_r$ for some $c_1 \in (1,9/8)$. So
\begin{align*}
D= c_1 \cdot p_r \cdot p_{r-1} \cdots p_1 & = c_1 \cdot p_r \cdot \left(p_r e^{p_r}\right) \cdot\left(p_r e^{p_r e^{p_r}}\right) \cdots \left(p_r \underbrace{e^{p_r e^{... e^{p_r}}}}_{\substack{\text{exponential tower}\\\text{of height }r-1}}\right)\\
		& = c_1 \cdot p_r^r \cdot \left(e^{p_r} e^{p_r e^{p_r}} \cdots \underbrace{e^{p_r e^{... e^{p_r}}}}_{\substack{\text{exponential tower}\\\text{of height }r-1}} \right).\\
\end{align*}
Applying the logarithmic function to both sides $(r-1)$ times, we have $p_r\leq \log_{(r-1)} D = p_r + O(\log p_r)$, implying that $\log_{(r-1)}D-c\log_{(r)}D\leq p_r\leq \log_{(r-1)}D$ for some constant $c>0$. We see from the above that $p_{r-1}=p_re^{p_r}$ and that $p_i\geq p_rp_{i+1}$. It follows that $\ell_i\leq \ell_rp_r^{i-r}$ and so 
\[
\ell_r = D(1-\eta)\text{ for }\eta \leq \frac{2}{\log_{(r-1)}D} \ra 0.
\]
 Let $T_i : = \ell_i\log\frac{\ell_i}{\ell_{i-1}}$. Then we have 
\[
\frac{T_i}{T_{i-1}}=p_i\frac{\log p_i}{\log p_{i-1}}=\frac{p_i\log p_i}{p_i+\log p_r}\geq \tfrac12\log p_i>1.
\]
It follows that for all $i<r-1$, we have $T_i <T_{r-1}$. Now $\log p_{r-1} = p_r + \log p_r$ implies that 
\[
T_{r-1} = \ell_{r-1}\log p_{r-1}=\ell_{r-1}(p_r + \log p_r)\leq \frac{\ell_r}{p_r}(p_r + \log p_r)  = O(\ell_r) = O(D).
\]
Thus, the objective is dominated by last summand $T_r = \ell_r\log\frac{\ell_r}{\ell_{r-1}}$, resulting in minimum value of at least $D\log_{(r)}D+O(D)$.
 \end{proof}

\subsection{Upper bound on $\D(G^r)$}\label{sec3.1}
Let 
\[
D_* = \frac{\log n}{\log_{(r+1)} n}.
\]
We prove that $\D(G^r) \leq D_*(1+\e)$ w.h.p., where $\e=\frac{1}{\log_{(r+1)}n}$.  
Using $u_{\ell_1,...,\ell_r}$ as in~\eqref{bound} and Lemma~\ref{lem2},
\begin{align*}
	\log u_{\ell_1,...,\ell_r} 
	&= D_* \log d - D_*(D_*-\ell_r) - \sum\limits_{i=1}^r \ell_i \log\frac{\ell_i}{\ell_{i-1}} -\tfrac12\log L+O(r) \\
	&\leq O(D_*)-\sum\limits_{i=1}^r \ell_i \log\frac{\ell_i}{\ell_{i-1}} \\
	& \leq O(D_*)-D_*(1+\e) \log_{(r)} (D_*(1+\e))\\
	&\leq -\brac{1+\frac{\e}2}\log n.
\end{align*}

Hence $u_{\ell_1,...,\ell_r} \leq \frac{1}{n^{1+\e/2}}$. Since there are at most $D_*^r$ terms in the summation, we have from~\eqref{uhu} that
$$\Pr\brac{\sum\limits_{i=1}^r d_i(v) = D} =(1+o(1))\sum_{\substack{\ell_1,...,\ell_r \\ \ell_1+...+\ell_r = D}} u_{\ell_1,...,\ell_r}\leq \frac{2D_*^r}{n^{1+\e/2}}.$$

Finally, by taking the union bound over all $n$ vertices, 
$$\Pr(\D_r(G) \geq D_*)\leq \sum_{i=1}^n \Pr\brac{ \sum\limits_{j=1}^r d_j(v_i) = D_*} \le \frac{2nD_*^r}{n^{1+\e/2}}  \ra 0.$$

\subsection{Lower bound on $\D(G^r)$}
We now use the second moment method to show that $\D_r(G) \geq D_*(1-\e)$ w.h.p. For $i=1,2,...,n$, let $X_i$ be the indicator random variable for $\sum\limits_{j=1}^r d_j(v_i) > D_{*}(1-\e)$, i.e.~the event that $v_i\in V(G)$ has degree greater than $D_*(1-\e)$ in $G^r$. Let $X_* = \sum\limits_{i=1}^n X_i$ and ${\boldsymbol \ell_*}$ be the sequence of values of $\ell_i$ which achieve the lower bound in Lemma~\ref{lem2}. 
\begin{align} 
\Pr[X_i = 1] \geq u_{\boldsymbol \ell_*} 	& = \exp\set{-D_*(1-\e)\log_{(r)} D_*(1-\e)+ O(D_*)} \nn\\
& \geq \exp\set{-(1-\e)\log n + O(D_*)}\nn\\
& \geq \frac{1}{n^{1-\e/2}}.\label{extra}
\end{align}
Then we have 
\[
\E[X_*]  \geq n^{\e/2}. 
\]
On the other hand,
$$
\begin{aligned} \Pr[X_* > 0] & \geq \frac{\E[X_*]^2}{\E[X_*^2]}  = \frac{\E[X_*]^2}{\displaystyle \sum_{i=1}^n \E\brac{X_i^2} + \sum_{i \neq j} \E\brac{X_iX_j}}\\
& \geq \frac{\E[X_*]^2}{\displaystyle \E\brac{X_*} + \E\brac{X_*} \sum_{j : d(v_1, v_j) > r} \E\brac{X_j }+\E\brac{|\set{j : d(v_1, v_j) \leq r}|}}.
\end{aligned}
$$
Here we use the fact $X_j\leq 1$ for all $j$ and that the only variables $X_j$ that are dependent on $X_1$ are for the vertices $v_j$ within a distance $r$ of $v_1$.

Now,
\[
\E\brac{|\set{j : d(v_1, v_j) \leq r}|}\leq (10d\log n)^{r+1}+n\Pr(\D(G_{n,p}\geq 10d\log n)=O(\log^{r+1}n).
\]
So
\[
\Pr[X_* > 0] \geq \frac{\E[X_*]^2}{\E[X^*]+\E[X_*]^2+O(\log^{r+1}n)}\geq \frac{1}{n^{-\e/2}+1+O(n^{-\e}\log^{r+1}n)}\to 1.
\]
\section{Proof of Theorem~\ref{th2}}\label{chi2}
\begin{proof}
For ease of reading we will denote $G_{n, p}$ with $G$ and $\Delta(G_{n, p})$ with $\Delta$ throughout the proof. Thus, we would like to show that w.h.p.~$\chi(G^2) = \Delta + 1$.
It will be important to keep in mind the known result that $\Delta \sim \log n/\log\log n$ w.h.p.
\paragraph{Lower bound.} There exists a vertex $v \in [n]$ with degree $\Delta$ in graph $G$. Let $N(v)$ denote its neighbours in $G$. Then $\{v\} \cup N(v)$ forms a clique of size $\Delta + 1$ in $G^2$. Thus $\chi(G^2) \geq \Delta + 1$. 

\paragraph{Upper bound.}
Let $S$ be the set of vertices that have degree greater than $\Delta$ in $G^2$ (or equivalently vertices that have more than $\Delta$ other vertices in their distance-2 neighbourhood in $G$). We prove the upper bound in two claims.

 The following claim is also needed in Section~\ref{chir}, where $r\geq 2$ rather than $r=2$. Given $v\in[n]=V(G)$ where $G =  G_{n,p}$, define $N_t(v)$ as the set of vertices at distance at most $t$ from $v$ in $G$. Let $N_r(S) = \bigcup\limits_{v\in S} N_r(v)$. Also, let $\D_{s}$ denote the maximum degree in $G^s$.
\begin{claim}\label{claim:chr}
	If $G[S\cup N_r(S)]$ is a forest, then $G^r$ can be vertex-colored with $\Dr+1$ colors. 
	\end{claim}
\begin{proof}
	We first claim that $S$ can be colored in $G^r$ with $\Dr+1$ colors. The neighbors of $S$ in $G^r$ are precisely $N_r(S)$, hence while coloring $S$ we can restrict our attention to the forest $G[S\cup N_r(S)]$. Assign a root to each tree in the forest. We color the $S$ vertices in this forest in a top-down approach. Assume that the top $k$ layers in a Breadth First Search starting with the root of a tree in $G[S\cup N_r(S)]$ are colored such that the coloring is proper in $G^r$. Consider a vertex $v\in S$ in the $(k+1)$-th layer. All the colored vertices which are up to distance $r$ from $v$ must be up to distance $r-1$ from its parent $u$ in the $k$-th layer. Since $\Dr$ is the maximum degree in $G^{r-1}$, $u$ has at most $\Dr-1$ already colored neighbors other than $v$ within distance $r-1$ in $G$. Thus, $v$ has at most $\Dr$ already colored neighbors among $S\cup N_r(S)$ in $G^r$ and we have an available color among $\Dr+1$ colors to assign to $v$. We thus conclude that $S$ can be colored with $\Dr+1$ colors in $G^r$. Now by the definition of $S$, all vertices outside $S$ have degree at most $\Dr$ in $G^r$. We can thus greedily color all the remaining vertices with $\Dr+1$ colors.
\end{proof}
\begin{claim}\label{claim:forest}
    $G[S \cup N(S)]$ is a forest w.h.p.
\end{claim}
\begin{proof}
Suppose that $s,t=O(1)$. The number $Z_{s,t}$ of vertices within distance $s$ of a cycle of length at most $t$ in $G$ satisfies 
\beq{Zst}{
\E(Z_{s,t})\leq \sum_{i=0}^{s+t}n^ii!p^i =O(d^{s+t})=O(1).
}
\noindent So Markov's inequality implies that $Z_{s,t}\leq \log \log n$ w.h.p. This implies that w.h.p.~any vertex on a cycle of length at most 6 has degree at most $\log\log n$.

We will compute an upper bound on the probability of a cycle occurring in $G[S \cup N(S)]$. Recall that vertices in $S$ have more than $\Delta$ vertices in their distance-2 neighbourhood. Thus, if a vertex is in $S$ \emph{or} in $N(S)$, it must have at least $\Delta$ vertices in its distance-3 neighbourhood in $G[S \cup N(S)]$.

Suppose, for the sake of contradiction, that $S \cup N(S)$ does not induce a forest. Then it must induce some cycle. By what we have shown, this cycle must be longer than $6$ w.h.p.

Consider an induced cycle in the induced graph $G[S \cup N(S)]$ and let $\ell$ be its length. Note that at every vertex there must be at least $\Delta$ vertices in the distance-3 neighbourhood on the induced graph. By the minimality of the cycle, it has $\lfloor \ell/ 6 \rfloor$ vertices $S'$ at distance at least 6 from each other in the induced graph. Note that in the induced graph, any vertex can only be distance 3 from at most one vertex in $S'$.

Next let $q$ denote the probability that vertex 1 has at least $\D_1=\D-6$ vertices at distance at most 3 in $G$. Then, let $k_i,i=1,2,3$ is the number of vertices in level $i$ of the breadth first search tree that contains these $\D_1$ vertices. Note that the absence of a 6-cycle implies that $k_3\geq (k_1+k_2+k_3)/3$. Then
\begin{align*}
q
&\leq \sum_{K=\D_1}^{3\D_1}\sum_{\substack{k_1+k_2+k_3=K\\k_3\geq \D_1/3}}\binom{n}{k_1,k_2,k_3,n-1-K}\bfrac{d}{n}^{k_1}\brac{1-\brac{1-\frac{d}{n}}^{k_1}}^{k_2}\brac{1-\brac{1-\frac{d}{n}}^{k_2}}^{k_3}\\
&\leq \sum_{K=\D_1}^{3\D_1}\sum_{\substack{k_1+k_2+k_3=K\\k_3\geq \D_1/3}}\frac{n^K}{k_1!k_2!k_3!}\bfrac{d}{n}^{k_1}\bfrac{dk_1}{n}^{k_2}\bfrac{dk_2}{n}^{k_3}\\
&\leq \sum_{K=\D_1}^{3\D_1}(de)^K\sum_{\substack{k_1+k_2+k_3=K\\k_3\geq \D_1/3}}\frac{1}{e^{k_3}k_3!}\\
&\leq \sum_{K=\D_1}^{3\D_1}(de)^K\binom{K-1}{2}\frac{1}{e^{K/3}(K/3)!}=o(1),
\end{align*}
since the term $(K/3)!$ grows faster than $(de)^K$.

We conclude that the probability that $G[S \cup N(S)]$ contains a cycle and the shortest cycle has length at least 6 is at most
\[
\sum_{6 \leq \ell \leq n} \mathbb{P}(S \cup N(S) \text{ contains an induced $\ell$-cycle})\leq \sum_{6 \leq \ell \leq n} n^{\ell} \cdot \bfrac{d}{n}^{\ell} \cdot q^{\rdown{\ell/6}}\leq d^6\sum_{m=1}^\infty(d^6q)^m=o(1).
\]
The factor $q^{\rdown{\ell/6}}$ arises from the fact that the vertices in $S'$ are at distance at least 6.
\end{proof}
The combination of Claim~\ref{claim:chr} (with $r=2$) and Claim~\ref{claim:forest} completes the proof of the upper bound.
\end{proof}

\section{Proof of Theorem~\ref{th3}}\label{chir}
We follow the same path as in Section~\ref{chi2}, though the calculations are  slightly more involved.  Let $S$ be the set of all vertices that have degree greater than $\Dr$ in $G^r$. For the lower bound we observe that the set of vertices at distance at most $\rdown{r/2}$ from a vertex $v$ form a clique in $G_{n,p}^{r}$. The upper bound in Theorem~\ref{th3} follows from Claim~\ref{claim:chr} and the following lemma.
\begin{lemma}\label{lem2a}
	$G[S\cup N_r(S)]$ is a forest w.h.p.
\end{lemma}
\begin{proof} 
Putting $s=r,t=8r$ in~\eqref{Zst} we see that w.h.p.~any vertex within distance $r$ of a cycle of length at most $8r$ has degree at most $\log\log n\ll \D_{r-1}$, by Theorem~\ref{th1}. This establishes that w.h.p.~the short cycles in $G$ (length up to $8r$) are at distance greater than $r$ from $S \cup N_r(S)$, the set of high degree vertices and their neighbors up to distance $r$. 
	
	We will now compute an upper bound on the probability of a cycle occurring in $G[S\cup N_r(S)]$. Recall that vertices in $S$ have more than $\Dr$ vertices in their distance-$r$ neighborhood. Thus, if a vertex is in $S \cup N_r(S)$, it must have at least $\Dr$ vertices in its distance-$2r$ neighborhood in $G[S\cup N_r(S)]$. Suppose, for the sake of contradiction, that $S \cup N_r(S)$ does not induce a forest. Then it must induce a cycle, which must be longer than $8r$ w.h.p. Consider a cycle $C$ of length $\ell > 8r$ in the induced graph $G[S\cup N_r(S)]$, on vertices $u_1, u_2,...,u_\ell$. For each $u_i$, let $v_i$ denote a vertex of $S$ which is closest to $u_i$ in $G[S\cup N_r(S)]$. For each $i=1,2,..,\ell$, let $A_i$ denote the set of vertices in $C$ within distance $2r$ of $u_i$ along $C$, namely $A_i = \{u_{i-2r},...,u_i,...,u_{i+2r}\}$ where indices are modulo $\ell$. If $N_r(v_i) \cap N_r(v_j) \neq \varnothing$ for some $u_j\notin A_i$, then we have a path of length at most $4r$ from $u_i$ to $u_j$ in $G[S\cup N _r(S)]$, namely $u_i\rightsquigarrow v_i \rightsquigarrow v_j \rightsquigarrow u_j$. Since $\ell > 8r$, we can replace the larger part of the cycle between $u_i$ and $u_j$ (strictly greater than $4r$ in length) with this path to reduce the length of the cycle. Since $G[S\cup N_r(S)]$ cannot have short cycles, eventually we obtain a cycle $C$ such that for any $u_i \in C$, $N_r(v_i)$ is disjoint from $N_r(v_j)$ for all $u_j \in C\sm A_i$. In this way we can divide $C$ into $\lfloor \frac{\ell}{4r}\rfloor$ disjoint arcs of length $4r$. The set $T$ of central vertices of these arcs satisfies the property that their respective closest neighbors from $S$ have disjoint neighborhoods up to distance $r$ in the induced graph $G[S\cup N_r(S)]$. Thus any vertex from $S\cup N_r(S)$ can be distance at most $2r$ away from at most one vertex from $T$. For any $u_i\in T$, call the path $u_i\rightsquigarrow v_i$ along with the induced graph on the neighborhood $N_r(v_i) \sm C$ the {\em neighborhood structure} of $u_i$. The neighborhood structures of all vertices in $T$ are pairwise disjoint. 
		
Now we obtain an upper bound on the probability that $G[S \cup N_r(S)]$ contains such a cycle of length $\ell > 8r$ along with the disjoint neighborhood structures at $\lfloor \frac{\ell}{4r} \rfloor$ vertices of $T$. Note that vertices in $T$ are at distance multiples of $4r$ along $C$ from other vertices in $T$. Hence, fixing which of $\{u_1,u_2,...,u_{\lfloor\ell/4r\rfloor}\}$ appears in $T$ fixes the other vertices of $T$.
We have
\begin{align*}
		&\sum_{8r\leq \ell \leq n} \Pr\brac{\text{$S\cup N_r(S)$ contains a cycle $C$ of length $\ell$ with $\lfloor \ell/4r \rfloor$ disjoint nbhd structures} } \\
		&\leq \sum_{8r\leq \ell \leq n} \binom{n}{\ell} \frac{(\ell-1)!}{2} \left(\frac{d}{n}\right)^\ell \cdot \Pr\brac{ \text{disjoint nbhd structures at $\lfloor\ell/4r\rfloor$ vertices $u_i \in C$}} \\
		&\leq \sum_{8r\leq \ell \leq n}d^\ell \lfloor\ell/4r\rfloor \cdot \Pr\brac{ \text{disjoint nbhd structures exactly at $u_{1+4r\cdot i}$, $\forall i=0,...,\lfloor \ell/4r \rfloor -1$}} \\
		& \leq \sum_{8r\leq \ell \leq n} d^{\ell} \lfloor \ell/4r\rfloor  \cdot \Pr\brac{\forall i = 0,...,\lfloor \ell/4r\rfloor -1,
		\begin{aligned}
			&  \text{ $\exists$ path $P_i : u_{1+4ri} \rightsquigarrow v_{1+4ri}$ of length $t_i\leq r$} \\ 
			& \text{and $|N_r(v_{1+4ri})| \geq \Dr-2r+t_i$}\\
			& \text{outside $C\cup P_i\cup \bigcup\limits_{j=0}^{i-1} N_r(v_j)$}
		\end{aligned}} \\
		&\leq \sum_{8r\leq \ell \leq n} d^{\ell} \lfloor \ell/4r\rfloor \cdot 
		\prod_{i=0}^{\lfloor \ell/4r \rfloor-1} \left(
		\begin{aligned} 
			& \sum_{t_i=1}^r \binom{n-\ell- \sum\limits_{k=0}^{i-1} |N_r(v_{1+4rk}) \sm C|}{t_i} t_i! \left(\frac{d}{n}\right)^{t_i} \times \\  
			& \Pr\brac{\text{$|N_r(v_{1+4ri})| \geq \Dr-2r+t_i$ outside $C \cup P_i \cup\bigcup\limits_{j=0}^{i-1} N_r(v_j)$}} 
		\end{aligned} \right) \\
		&\leq \sum_{8r\leq \ell \leq n} d^{\ell} \lfloor \ell/4r\rfloor \cdot 
		\prod_{i=0}^{\lfloor \ell/4r \rfloor-1} \left( \sum_{t_i=1}^r \binom{n}{t_i} t_i! \left(\frac{d}{n}\right)^{t_i} \times \Pr\brac{\text{$|N_r(v_{1+4ri})| \geq \Dr-2r+t_i$} }\right)\\
		&\leq2 \sum_{8r\leq \ell \leq n} d^{\ell} \lfloor \ell/4r\rfloor \cdot 
		\prod_{i=0}^{\lfloor \ell/4r \rfloor-1} \left(\brac{\frac{d^{r+1}-d}{d-1}} \times \frac{\sum\limits_{d = \Dr-2r}^{\D_r-r} d^r}{n^{1-\e/2}} \right),\qquad\text{$n^{1-\e/2}$ is from~\eqref{extra}}\\
		&\leq 2\sum_{8r\leq \ell \leq n} d^{\ell} \lfloor \ell/4r\rfloor \cdot 
		\left( \brac{\frac{d^{r+1}-d}{d-1}}\times \frac{\D_r^{r+1}}{n^{1-\e/2}} \right)^{\lfloor \ell/4r\rfloor} \\
		&\leq 8r \cdot d^{4r-1} \sum_{m=2}^\infty m \left(d^{4r}\cdot\brac{\frac{d^{r+1}-d}{d-1}} \cdot \frac{\D_r^{r+1}}{n^{1-\e/2}} \right)^{m}.
	\end{align*}
This approaches 0 as $n\rai$, since $\Dr \ll \log n \implies \frac{\D_r^{r+1}}{n^{1-\e/2}} \ra 0$.
\end{proof}
\section{Proof of Theorem~\ref{th4}}\label{sec:maindense}
In the proof, we will need to bound the density of vertex neighbourhoods in $G_{n,p}^r$ and apply the following theorem by Alon, Krivelevich and Sudakov~\cite{AKS99}. We note a recent improvement in the leading constant for this result in~\cite{DKPS20+}. 
\begin{theorem}[\cite{AKS99}, cf.~\cite{DKPS20+}] \label{thm:sparsenbhd}
There exists a constant $c$ such that for every graph $H$ of maximum degree $\Delta$ and real number $t$ it holds that if
\begin{itemize}
    \item $2 \leq t \leq \Delta$, and
    \item for all $v \in V(H)$ the neighbourhood of $v$ spans at most $\Delta^2/t$ edges,
\end{itemize}
then $\chi(H) \leq c\Delta / \log t$.
\end{theorem} 
Again, we denote $G_{n, p}$ by $G$ and $\Delta(G_{n,p})$ by $\Delta$ for simplicity. 

It is known that w.h.p.~$p=\om\log n$ implies that w.h.p.~all degrees in $G$ are $\sim np$, see for example Theorem 3.4(ii) of~\cite{FrKa16}.  (Here $A\sim B$ means that $A=(1+o(1))B$ as $n\to \infty$ and $A\lesssim B$ means $A\leq (1+o(1))B$.) 
\subsection{Upper bound}
Fix $v\in [n]$ and let $N_i=N_i(v)$ denote the vertices at distance $i$ from $v$ in $G$ and let $N(v)$ denote the neighbors of $v$ in $G^r$. Then,
\beq{up1}{
|N(v)| =\sum_{i=1}^r|N_i(v)|\leq  \sum_{i=1}^r\D^i \sim \D^r.
}
On the other hand, we have from the Chernoff bounds, that with $\th=1/(np)^{1/2}$ and $i\geq 2$,
\beq{up2}{
&\Pr\brac{|N_i(v)|\leq (1-\th)np|N_{i-1}(v)|\;\bigg|\; |N_{i-1}(v)|\geq \bfrac{np}{2}^{i-1}}\nonumber\\
&\leq \exp\set{-\frac{\th^2(n-\D^{i-1})(np/2)^{i-1}p}{4}}=o(n^{-1}).
}
So w.hp., the upper bound in~\eqref{up1} is also an asymptotic lower bound. We need two more facts. Let $\n_0=d^{1-\e}$.
\begin{align}
&\text{W.h.p., for all $v\in[n], 1\leq i\leq r$ and $w\notin N_i(v)$, $|\set{u\in N_i(v):\set{u,w}\in E(G)}|\leq \n_0.$}\label{up3}\\
&\text{W.h.p., for all $v\in[n]$ and $w\in N(v)$, $|\set{u\in N(v):\set{u,w}\in E(G)}|\leq \n_0$.}\label{up4}
\end{align}
Indeed,
\begin{align*}
\Pr(\neg\eqref{up3})&\leq n^2\binom{\D^i}{\n_0}p^{\n_0}\leq n^2\bfrac{\D^ipe}{\n_0}^{\n_0}\leq n^2\bfrac{2d^{r+\e}e}{n}^{d^{1-\e}}\\&\leq n^2\bfrac{2en^{(r+\e)(1/r-\e)}}{n}^{d^{1-\e}}=o(1).\\
\Pr(\neg\eqref{up4})&\leq n^2\binom{\D^r}{\n_0}^{\n_0}p^{\n_0}\leq n^2\bfrac{2d^{r+\e}e}{n}^{\n_0}=o(1).
\end{align*}
It follows from~\eqref{up3} and~\eqref{up4} that the number of paths of length at most $r$ in $G^r$ that join pairs of vertices in $N(v)$ is at most $\D^{2r-1}\n_0$. Note that when we trace a path between a pair of vertices, there are always at most $\D$ choices for the next vertex and there is at least one time where the number of choices is at most $\n_0$. Because $\D/\n_0\gtrsim d^{\e/2}$, we can apply Theorem~\ref{thm:sparsenbhd} with $t\sim d^{\e/2}$. Thus w.h.p., 
\[
\chi(G^r)\lesssim \frac{2cd^r}{\e\log d}.%\qedhere
\]
%\end{proof}
\subsection{Lower bound}
Let $X_k$ denote the number of  $k$-sets in $G^r$ for which no pair of vertices are connected by a path of length $r$. We show that $X_k=0$ w.h.p., for $k\geq k_0=\a n\log d/d^r$ where $\a=10r!$. Now
\beq{Janson}{
\E(X_{k_0})\leq \binom{n}{k_0}\exp\set{-\frac{\m^2}{2(D+\m)}}
}
where, given a fixed $k_0$-set, the expected number of number of paths $\m$ satisfies
\[
\m=\binom{k_0}{2}\binom{n-2}{r-1}p^r\sim \frac{\a d^r\log d}{2(r-1)!}\text{ and }D=\sum_{P\sim Q}\Pr(P,Q),
\]
where $P,Q$ run over all distinct pairs of paths of length $r$ that (i) join  a pair of vertices in $[k_0]$ and (ii) share at least one edge.

Inequality~\eqref{Janson} is an example of Janson's inequality~\cite{Jan90}. Now denote by $\cQ$, the set of paths $Q$ that share an edge with the path $P_0=(1,2,\ldots,r+1)$ in $K_n$. Then, where in the following, $i$ denotes the number of edges that a path $Q$ shares with $P_0$ and $j$ denotes the number of common vertices, 
\[
D=\m\sum_{Q\in \cQ}\Pr(Q\mid P_0)\leq \frac{\m}2\sum_{i=1}^{r-1}\sum_{j=i+1}^{r}\binom{r}{i}\binom{k_0}{r+1-j}p^{r-i}\leq \frac{r\m}2\sum_{i=1}^{r-1}\binom{k_0}{r-i}p^{r-i}\leq r\m k_0p\leq \m.
\]
Applying~\eqref{Janson},
\[
\E(X_{k_0})\leq \bfrac{ne}{k_0}^{k_0}\exp\set{-\frac{\m}{4}}\leq \brac{\frac{ed^r}{\a\log d}\cdot d^{-\a/4}}^{k_0}=o(1).
\]
Thus, w.h.p., there are no independent sets of size $k_0$ in $G^r$, in which case we have that w.h.p.
\[
\chi(G^r)\geq \frac{n}{k_0}=\frac{d^r}{4r\log d}.
\]
This completes the proof of Theorem~\ref{th4}.\hfill\qed
\section{Concluding remarks}
As hinted at in the discussion preceding their statements, Theorems~\ref{th3} and~\ref{th4} might be representative of two nontrivial regimes in the behaviour of $G_{n,p}^r$, one in which the chromatic number is locally determined, and the other in which it is more globally determined.

Writing $\omega(G)$ for the clique number of a graph $G$, we pose the following conjecture to make concrete this allusion.

\begin{conjecture}\label{conj:main}
Let $p=d/n$, where $d>0$ is a constant and let $r\geq 1$ be a fixed positive integer.  Then, w.h.p.~$\chi(G_{n,p}^r) = \Theta(\max\{\omega(G_{n,p}^r),n/\alpha(G_{n,p}^r)\})$.
\end{conjecture}

\noindent
The second term in the maximization coincides with the bound of Theorem~\ref{th4}. The first corresponds with the bound in Theorem~\ref{th3}, as well as the (higher) densities for which $G_{n,p}$ has diameter at most $r$ w.h.p. The main challenge in this conjecture is to determine the transitions between these three density regimes, which we leave to future investigations. It would also be interesting to find the leading constant implicit in the $\Theta$ (assuming the conjecture holds), in particular whether it could be (asymptotically) 1.

Our conjecture suggests that, at least  in the sparse regime, the correct value is closer to the lower bound in Theorem~\ref{th3}, due to the following observation.
\begin{lemma}\label{omom}
Let $p=d/n$, where $d>0$ is a constant and let $r\geq 1$ be a fixed positive integer. Then, w.h.p.~$\omega(G_{n,p}^r)\leq \Delta(G_{n,p}^{\rdup{r/2}})+1$.
\end{lemma}
\begin{proof}
Let $W$ denote a maximum clique in $G^r_{n,p}$. Note that $\D(G^\rdup{(r-1)/2}_{n,p})\leq |W|\leq \D(G^r_{n,p})+1=o(\log n)$.  We argue now that since $|W|=s_0=o(\log n)$, w.h.p.~$W$ contains at most $s_0$ edges. The probability that our assertion fails to be true is at most
\beq{EQ1}{
  \binom{n}{s_0}\binom{\binom{s_0}{2}}{s_0+1}\bfrac{d}{n}^{s_0+1}\leq \bfrac{ne}{s_0}^{s_0}\bfrac{s_0e}{2}^{s_0+1}\bfrac{d}{n}^{s_0+1}
= \frac{e^{2s_0+1}d^{s_0+1}s_0}{2^{s_0+1}n}=o(1).
}
We next observe that $W$ must contain at least one vertex of degree at least
\[s_1=\D(G^\rdup{(r-1)/2}_{n,p})^{1/r}/2.\] Then we claim that $G_{n,p}$ contains no vertex $v$ of degree at least $s_1$ that is within distance $r$ of a cycle of length at most $2r$. The probability that our assertion fails to be true is at most
\beq{EQ2}{
n\cdot\sum_{s\geq s_1}\binom{n-1}{s}p^s\cdot s\cdot \sum_{k=1}^rn^kp^k\cdot\sum_{\ell=3}^{2r}n^{\ell-1}p^\ell\leq 2rd^{3r}\sum_{s\geq s_1}\frac{d^s}{(s-1)!}=o(1).
}
We reason for this as follows: $n$ choices for $v$; then sum over $s$ a bound on the probability of $s$ neighbors; then $s$ choices for the start of a path to a cycle $C$; then sum over $k$ a bound on the probability of there being a path of length $k$ to $C$; then sum over $\ell$ a bound on the probability of a cycle $C$. 

It follows from \eqref{EQ1} and \eqref{EQ2} that $W$ induces a tree in $G_{n,p}$. Following this we can use the fact that in a tree, the radius is $\rdup{D/2}$, where $D$ is the diameter. So, we let $T$ be any spanning tree of $W$ and take $D=r$.
\end{proof}
\paragraph{Open access statement.} For the purpose of open access,
a CC BY public copyright license is applied
to any Author Accepted Manuscript (AAM)
arising from this submission.

\bibliography{references}
\bibliographystyle{abbrv}
\end{document}